\documentclass[12pt]{amsart}%
\usepackage{kotex}
\usepackage{amssymb}
\usepackage{amsmath}
\usepackage{amsfonts}
\usepackage{graphicx}%
\newtheorem{theorem}{Theorem}[section]

\newtheorem{proposition}[theorem]{Proposition}
\newtheorem{corollary}[theorem]{Corollary}

\theoremstyle{definition}
\newtheorem{definition}[theorem]{Definition}

\theoremstyle{remark}
\newtheorem{remark}[theorem]{Remark}
\numberwithin{equation}{section}

\renewcommand{\b}{\beta}
\renewcommand{\d}{\delta}

\newcommand{\G}{\Gamma}

\renewcommand{\l}{\lambda}

\renewcommand{\o}{\omega}
\renewcommand{\O}{\Omega}

\newcommand{\pa}{\partial}
\renewcommand{\r}{\rho}

\newcommand{\Th}{\Theta}
\renewcommand{\th}{\theta}

\newcommand{\we}{\wedge}

\begin{document}

\title[CR invariant subvarieties  ]{ Nullity of the  Levi-form and the associated subvarieties for pseudo-convex CR structures of hypersurface type }

\author[K. Chung]{Kuerak Chung}
\author[C.-K. Han]{Chong-Kyu Han}

\address[K. Chung]{ Korea Institute for Advanced Study, 85 Hoegi-ro, Dongdaemun-gu, Seoul 02455, Republic of Korea}
\email{krchung@kias.re.kr}

\address[C.-K. Han]{Department of Mathematics, Seoul National University, 1 Gwanak-ro, Gwanak-gu, Seoul 08826, Republic of Korea}
\email{ckhan@snu.ac.kr}

\thanks{The authors were  partially supported by National Research Foundation
of Korea with grant   NRF-
2017R1A2A2B4007119.       }


\begin{abstract}
Let $M^{2n+1}$, $n\ge 1$, be a smooth manifold with a pseudo-convex integrable CR structure  of hypersurface type. We consider a sequence of CR invariant subsets $ M=\mathcal S_0 \supset \mathcal S_1 \supset \cdots \supset \mathcal S_{n}, $  where $\mathcal S_q$ is the set of points where the Levi-form has nullity $\ge q$. We prove  that $\mathcal S_q$'s are  locally given as  common zero sets of the coefficients $A_j,$ $j=0,1,\ldots, q-1,$  of the characteristic polynomial of the Levi-form.  Some sufficient conditions for local existence of complex submanifolds  are presented in terms of the coefficients $A_j$.  
\end{abstract}

 \maketitle

\section*{Introduction  }

Let $M, $ or to specify the dimension $M^{2n+1}, $ $n\ge 1, $   be  a smooth ($C^{\infty}$)  manifold equipped with  a pseudo-convex integrable CR structure $(H(M),J)$ of hypersurface type.   We consider a sequence of CR-invariant subsets
\begin{equation*}\label{sequence} M=\mathcal S_0 \supset \mathcal S_1 \supset \cdots \supset \mathcal S_{n}, \end{equation*}  
where $\mathcal S_q$ is the set of those points where the Levi-form has nullity greater than or equal to $q$.  In \S 2 we prove that $\mathcal S_q$ is  locally given as a common zero set of the coefficients of the characteristic polynomial of the matrix representation of the Levi-form.  A complex submanifold of $M$ of complex dimension $q$,  $1\le q \le n, $ is a local embedding $f $ of an open subset 
$\mathcal O\subset \mathbb C^q $  that satisfies 
\begin{equation}\label{Cauchy-Riemann} \begin{aligned} 
df(T_x(\mathcal O)) &\subset H_{f(x)}(M), ~~ \forall x\in \mathcal O ~~ \quad  \text{ and  }&\\
 J\circ df &= df\circ J_{st}, \end{aligned}\end{equation} where $J_{st}$ is the standard complex structure tensor of $\mathbb C^q$.  Solving (\ref{Cauchy-Riemann}) for $f$ is an over-determined problem and  there is no solution generically. If such $f$ exists then its image must be contained in $\mathcal S_q.$ Hence a necessary condition for the existence of a complex manifold of dimension $q$ is  that  $\mathcal S_q $ has real dimension $\ge 2q.$  
 In \S 3 we present  some sufficient conditions for the existence of real defining functions for a complex submanifold of dimension $q.$  These are conditions on the coefficients of the characteristic polynomial of the levi-form that define $\mathcal S_q$.  In several complex variables  the function theory of a domain often depends on the geometry of the boundary. In particular,   at a boundary point $x$ of a pseudo-convex domain with smooth real-analytic boundary the subelliptic estimate  for the $\bar\pa$-Neumann problem holds for $(p,q)$-forms  if  there is no germ at $x$ of  a complex variety of dimension greater than or equal to $q$ in the boundary (see \cite{K} and \cite{DF}).   For real hypersurfaces in general, without assuming pseudo-convexity, the existence of a germ of complex hypersurface  is an  obstruction to the extension of holomorphic  functions of one side of $M$ to the other side (see \cite{T}).  The existence of complex submanifolds in real hypersurfaces has been studiied in \cite{AH}, \cite{HT} and \cite{NQD} without assuming the pseudo-convexity.  The present paper makes essential use of the fact that the eigen-values of the Levi-form are non-negative for pseudo-convex CR manifolds.  
The coefficients $A_j$ of the characteristic polynomial of the Levi-form depends on the choice of local basis of $(1,0)$ vectors,  they are invariant only under unitary change of basis.  However,  $\mathcal S_q$  the common zero set of $A_j,$ $j=0,1,\ldots, q-1,$ 
is a CR invariant,  which  is interesting  from the viewpoint of the invariant theory.  We work in $C^{\infty}$ category.  All the manifolds  in this paper shall be assumed to be connected and oriented, or  to be a connected submanifold sitting in  a small neighborhood $U\subset M.$   The authors thank  Dmitri Zaitsev for the discussions we had while he was visiting KIAS.
We thank  Sungyeon Kim  for the valuable comments and  suggestions that she gave us.

\vskip 1pc

\section{Pseudo-convexity and the nullity of the Levi-form }
\vskip 1pc

 Let $ M$  be a smooth ($C^{\infty}$) manifold of  dimension $2n+1,$  $n\ge 1. $  A CR structure of hypersurface type, or simply a CR structure,  on $M$ is a pair 
$(H(M), J), $ where $H(M)$ is a 
smooth subbundle  of codimension $1$ of the tangent bundle  $T(M) $  and $J$ is  an almost complex structure
 on $H(M)$,  that is,  $J: H(M)\longrightarrow H(M)$ is a smooth  bundle isomorphism  for which 
$$\begin{matrix} H(M) & & \overset{J}\longrightarrow&&H(M)\\ 
& \searrow  &&\swarrow &\\ & &M &  \end{matrix}$$
commutes and 
\[
 J^2 = - I_{2n} .
 \]

Then the complexification $\mathbb C\otimes H(M)$ has a decomposition  
\begin{equation*}\label{decompose1} \mathbb C\otimes H(M) = H^{1,0}(M) \oplus H^{0,1}(M), \end{equation*} 
where $H^{1,0}(M)$ (resp.  $H^{0,1}(M)$) is  the set of eigen-vectors of $J$ corresponding to the eigen-value
$\sqrt{-1}$ (resp. $-\sqrt{-1}$).  Locally,  there exist real vector fields $X_1, \ldots, X_n$ so  that $X_j, JX_j, $  $j=1, \ldots, n,  $ span 
$H(M).$ Then 

\begin{equation}\label{10basis} L_j:= X_j - \sqrt{-1} JX_j , ~~j=1, \ldots, n \end{equation}
span
 $H^{1,0}(M)$ and 
\begin{equation}\label{01basis} \bar L_j:= X_j + \sqrt{-1} JX_j , ~~j=1, \ldots, n \end{equation} span $H^{0,1}(M)$, respectively.
A diffeomorphism $F$ of  $M$ onto a CR manifold of the same dimension  $\widetilde M$  is called a CR mapping if $F$ preserves the CR structure, namely,  $dF$ maps $H(M)$ onto $H(\widetilde M)$ and 
\begin{equation}\label{CR equiv1} dF\circ J = J \circ dF
\end{equation} on $H(M)$.  The CR structure bundle $H(M)$ is locally given by a non-vanishing real $1$-form that annihilates $H(M)$:   any point $x\in M$ has a neighborhood $U$ on which there is a smooth non-vanishing real $1$-form $\th$ such that 
\begin{equation}\label{theta} H(U) = \th^{\perp}:=\{V\in T(U): \th(V)=0\}.
\end{equation}
We fix $U$ for our local arguments
assuming the local bases (\ref{10basis}) and (\ref{01basis}) are  defined on $U$.
  By $\O^0 (U) = C^{\infty}(U)$ we denote the ring of smooth complex valued functions  on $U$ and by $\O^p(U)$ the module over $\O^0(U)$ of smooth $p$-forms and by
\begin{equation*}\label{EDS} \O^*(U) =\bigoplus_{p=0}^{2n+1} \O^p(U)\end{equation*}
the exterior algebra of smooth differential forms on $U$ with complex  coefficients. 
 A subalgebra $\mathcal I\subset \O^*(U)$  is called an ideal if the following conditions hold: 
 \vskip 1pc
\noindent {\it i)} 
$ \mathcal I\we \O^*(U) \subset \mathcal I $ 

\noindent{\it ii)} if  $\phi=\sum_{p=0}^{2n+1} \phi_p \in \mathcal I,  ~ \phi_p\in\O^p(U), $  then each $\phi_p\in \mathcal I~~ $   (homogeneity).
\vskip 1pc  Because of the homogeneity condition 
$\mathcal I $ is two-sided, that is,  $$\O^*(U)\we \mathcal I \subset \mathcal I.$$ 
We consider in this paper only those ideals that are generated by $0$-forms (functions) and $1$-forms:
for a finite set of smooth functions  ${\bf r}=(r^1, \ldots, r^d)$ 
and a finite set of smooth $1$-forms ${\bf \Th}=(\th^1, \ldots, \th^{s})$ let
$\mathcal I({\bf r}, {\bf \Th})$ be the ideal generated by $ {\bf r}$ and $ {\bf \Th}.$
For two elements $\phi, \psi \in \O^*(U)$  we write 
\begin{equation*} \phi \equiv \psi, \quad\text{mod } ({\bf r}, {\bf\Th}) \end{equation*} or
\begin{equation*} \phi  \overset{({\bf r}, {\bf \Th})} \equiv \psi, \end{equation*}
if and only if 
$$ \phi - \psi \in \mathcal I({\bf r}, {\bf\Th}).$$   The system $\bf r$ is said to be \emph{non-degenerate} if $dr^1\we\cdots\we dr^d\neq 0.$
For other definitions and notations concerning the  exterior algebra we refer the readers to \cite{BCG3}. 
The CR structure bundle $H(M)$ is  \emph{integrable}  in the sense of  Frobenius if
\begin{equation}\label{integrability} d\th \equiv 0, ~~~ \text{ mod } (\th), \end{equation}  
where $\th$ is as in (\ref{theta}).     By the Frobenius theorem if (\ref{integrability}) holds $M$ is locally foliated by integral manifolds of $H(M). $  If  this is the case  $J$ is an almost complex structure on each leaf.  
\begin{definition}\label{Levi} The \emph{torsion tensor} of the CR structure  is 
\begin{equation}\label{Levi2} d\th, \text{ mod } (\th).\end{equation}
\end{definition}  
We regard (\ref{Levi2}) as an element of the quotient module $\O^2(U)/(\th)$, where $(\th$) is the set of all  $2$-forms 
$$\{\th\we \o : ~~\o\in \O^1(U)\}.$$  From the viewpoint of (\ref{integrability}),  the torsion tensor  is the obstruction to the local integrability of $H(M)$. 
The torsion tensor of the CR structure  defines a hermitian form  $\mathcal L$,  which is  called the \emph{Levi-form,}  on $H^{1,0}(U)$ by 
\begin{equation}\label{Levi3} \mathcal L (L, L') := \frac{1}{\sqrt{-1}} d\th(L\we\bar L').\end{equation}   If $\mathcal L$ is semi-definite at $x$ the CR structure is said to be \emph{pseudo-convex} at $x$.   If $\mathcal L$ is semi-definite at every point of $M$ the CR structure is said to be pseudo-convex.

\begin{definition}\label{nullity} At a point $x\in M$ the \emph{null space} of $\mathcal L$ is 
\begin{equation}\label{null space} \mathcal N_x := \{L\in H^{1,0}_x(M) : \mathcal L(L, L)=0\}.\end{equation}   Notice that if $M$ is pseudo-convex then $ \mathcal N_x $ is a subspace.
The complex dimension of $\mathcal N_x$ is called the \emph{nullity} of the Levi-form at $x.$   \end{definition}   
 If $F$ is a CR mapping of $M$ onto a CR manifold of same dimension $\widetilde M$ then by \eqref{CR equiv1} 
$dF$ preserves the type of  vectors, that is, $dF$ extends to the isomorphism
\begin{equation*}\label{type preserv} \begin{aligned} H^{1,0}(M) &\overset{dF}\longrightarrow H^{1,0}(\widetilde M)\\
 H^{0,1}(M) &\overset{dF}\longrightarrow H^{0,1}(\widetilde M), \end{aligned}\end{equation*}
Thus we have 
\begin{proposition}\label{global CR} Let $(H(M), J)$ be a smooth pseudo-convex  CR structure on a smooth manifold $M^{2n+1}, $ $n\ge 1. $  Then the notions of the null space and the nullity of the Levi-form defined locally by (\ref{Levi3})-(\ref{null space}) do not depend on the choice of $\th, $   and therefore,  well defined globally.  The nullity of the Levi-form is invariant under CR mappings.    \end{proposition}
Now  for each $q=0, 1, \ldots, n, $  let 
$ \mathcal S_q $ be the set of points where the nullity of the Levi-form is greater than or equal to $q$. 
Then we have a sequence of CR-invariant subsets
\begin{equation*}\label{sequence} M=\mathcal S_0 \supset \mathcal S_1 \supset \cdots \supset \mathcal S_{n}. \end{equation*}  

\vskip 2pc
\section{Invariant subvarieties of pseudo-convex CR manifolds }
Consider the  matrix representation of the Levi-form $\mathcal L$ with respect to the basis 
(\ref{10basis}):  let 
\begin{equation}\label{Levi matrix} T:=\left[T_{jk}\right]_{n\times n}, \text {where } T_{jk}:=  \mathcal L(L_j, L_k).\end{equation}   $T$ is a hermitian matrix and assumed to be positive semi-definite.

\begin{theorem} Suppose that a smooth real manifold  $M^{2n+1} $ admits a pseudo-convex CR structure $(H(M), J).$
Let  $T$ be a matrix representation of the Levi-form  defined locally as in (\ref{Levi matrix})  and   
\begin{equation*}\label{characteristic} \text{det } (T-\l I) :=   \sum_{k=0}^{n-1} A_k (-\l)^k + (-\l)^n \end{equation*}
be the characteristic polynomial of $T$. Then for each $q= 1, \ldots, n, $  $\mathcal S_q$ is  given as a common zero set of $A_j$, for all $j $ with $0\le j\le q-1.$ 
\end{theorem} 

\begin{proof}  We fix a point $x\in M$ and a neighborhood $U$ of $x$.  By a unitary change of basis $\mathcal L$ can be diagonalized.   The pseudo-convexity implies that the diagonal elements are non-negative, and the number of zeros in the diagonal is the nullity of the Levi-form.  To be precise,  define a hermitian metric on $H^{1,0}_x(M)$ by declaring $L_j$'s as in (\ref{10basis})  are orthonormal.   By the spectral theorem, there exists an orthonormal basis $Q_1, \ldots, Q_n$  that are eigen-vectors of $T$.  Let $d_j$ be the eigen-value corresponding to $Q_j$.  Setting   $Q_j = \sum_{k=1}^n Q_j^k L_k $ and  $Q=[Q_j^k]$  we have
\begin{equation*}\label{QTQ*} QT(x)Q^* = QT(x)Q^{-1}= \left[\begin{matrix} d_1 &&&\\& d_2&&\\&&\ddots&\\&&&d_n
\end{matrix}\right] ,  d_j\ge 0, 
\end{equation*} so that 
\begin{equation*} \mathcal L( Q_j, Q_k) = \begin{cases}d_j,\text {if } j=k\\0, ~~~~~\text {if } j\neq k.\end{cases}
\end{equation*}  Suppose $x\in \mathcal S_1.$  Then the Levi-form  $\mathcal L(x)$  has a non-trivial null vector.    Setting it by
\begin{equation*} V = \sum_{j=1}^n b^j Q_j, \end{equation*} we have
\begin{equation}\label{null0} \begin{aligned} 0 &= \mathcal L(V, V)\\
&=\sum_{j=1}^n |b^j|^2 d_j.\end{aligned}\end{equation}
Since $d_j\ge 0,$ (\ref{null0}) implies that at least one of $d_j$'s  is  zero. Therefore, 
\begin{equation*} \mathcal A_0(x) = d_1\cdots d_n = 0.\end{equation*}

Now suppose $x\in \mathcal S_q.$ Then the Levi-form $\mathcal L(x)$  has at least  $q$ independent null vectors.    Setting them by
\begin{equation*} V_{\l} = \sum_{j=1}^n b_{\l}^j Q_j, \end{equation*} we have
\begin{equation}\label{null1} \begin{aligned} 0 &= \mathcal L(V_{\l}, V_{\l})\\
&=\sum_{j=1}^n |b_{\l}^j|^2 d_j, \quad  \l=1, \ldots, q.\end{aligned}\end{equation}
We restate  (\ref{null1}) in matrices  as 
\begin{equation}\label{null2}\underbrace{ \left[\begin{matrix} |b_1^1|^2 &\cdots& |b_1^n|^2\\
\vdots && \vdots\\
 |b_q^1|^2 &\cdots& |b_q^n|^2\end{matrix}\right]}_B  \left[\begin{matrix}d_1\\ \vdots  \\d_n\end{matrix}
\right] = \left[\begin{matrix}0\\ \vdots \\ 0\end{matrix}\right]\end{equation}
Since $V_{\l}$'s are independent 
\begin{equation}\label{null3} \beta:=  \left[\begin{matrix} b_1^1 &\cdots& b_1^n\\
\vdots && \vdots\\
b_q^1 &\cdots& b_q^n\end{matrix}\right]
\end{equation} has rank $q$, so that $\b$ has $q$ independent columns, which implies that $q$ columns 
of $B$ are non-zeros.  Since  $d_j\ge 0,  $ for all $j=1, \ldots, n, $  (\ref{null2}) implies that  at least $q$ of $d_j$'s  are zeros.  For $k=0, 1, \ldots, q-1, $  $A_k(x)$ is the symmetric polynomial of degree $n-k$ in $d_1, \ldots, d_n, $
therefore, $\mathcal A_k(x)=0. $ 

Conversely,  if $A_0(x)=\cdots = A_{q-1}(x)=0 $ then there are at least $q$ zeros in $\{d_1, \ldots, d_n\}, $ say
$$d_1=
\cdots = d_q=0.$$
Then $Q_1, \ldots, Q_q$ are null vectors of $\mathcal L(x).$ 
\end{proof}
\vskip 2pc

\section{Integrable CR structures}
The CR structure $(H(M), J)$ is said to be \emph{integrable} if   the module of smooth sections of $H^{1,0}(M)$,  denoted by 
$\G(H^{1,0}(M)),  $  is closed under the Lie bracket:
\begin{equation}\label{integrability1} [L, L']\in \G(H^{1,0}(M)), \quad \forall L, L'\in \G(H^{1,0}(M)).\end{equation}
(\ref{integrability1}) is equivalent  to 
\begin{equation}\label{Nijenheus}[JX,JY]=[X,Y] +J[JX,Y]+J[X,JY], \quad \forall X, Y\in \G(H(M)).\end{equation} In this section  $M$ is assumed to be a CR manifold with an integrable CR structure.  Then if a submanifold $\mathcal S $ of $ M $ is $J$-invariant, namely, if 
$ JT(\mathcal S) \subset T(\mathcal S),$ 
then $\mathcal S$ is a complex manifold with the complex structure $J$  by the following

\begin{theorem}\label{NN} (Newlander-Nirenberg \cite{NN})    Let $\mathcal S $ be a smooth real manifold of dimension $2m$ with a smooth almost complex structure $J$ that is integrable in the sense that  any vector fields $X$ and $Y$ defined locally on an open neighborhood $U\subset \mathcal S$  satisfy (\ref{Nijenheus}).  Then $\mathcal S$ is covered by coordinate patches with complex  coordinates in which the coordinates in overlapping patches are related by holomorphic transformations, so that for a local coordinate system $(z_1, \ldots, z_m), $ $z_j=x_j+\sqrt{-1}y_j,$ 
\begin{equation*}\label{J coord}\begin{aligned} J\left(\frac{\pa}{\pa x_j}\right)&=\frac{\pa}{\pa y_j}, \\
 J\left(\frac{\pa}{\pa y_j}\right)&= -\frac{\pa}{\pa x_j}, \quad j=1, \ldots, m.\end{aligned}\end{equation*}
\end{theorem}

For the definition of the tangential  Cauchy-Riemann operator $\bar\pa_b$ we refer the readers to \cite{CS}.   For a function $f$, $\bar\pa_b f$  is a section of $H^{*0,1}(M):=(H^{0,1}(M))^*  $ defined by  
\[ (\bar\pa_b f)(V) = df (V), \quad\text{for any~ }  V\in\G(H^{0,1}(M)).\]
Let $\pi_{1,0}$, $\pi_{0,1}$ and $\pi_{t}$ be the projections of $\mathbb C\otimes T^*(M)$ onto each component of 
\begin{equation*}\label{CT*M}
\mathbb C\otimes T^*(M) = H^{*1,0}(M) \oplus H^{*0,1}(M) \oplus \langle \th\rangle .
\end{equation*}
 Then  
\begin{equation*}\label{dbarb}\begin{aligned} \bar\pa_b f  &:= \pi_{0,1} ~df \\
\pa_b f &:= \pi_{1,0} ~df . \end{aligned}
 \end{equation*}

 \vskip 1pc

  \begin{theorem}\label{complex manifold} Suppose that $M^{2n+1}$ is a smooth manifold with  integrable CR structure $(H(M), J)$.  For an integer $k, $  $0\le k \le n, $ let $\r_1, \ldots, \r_{2k+1}$ be a non-degenerate system of real-valued functions defined locally on an open subset $U\subset M. $   Let  $\mathcal S $ be the common zero set of $\r_j, $ $j= 1, \ldots, 2k+1.  $   
Then $\mathcal S $ is  a complex manifold  of complex dimension $n-k$  with $J$ as its complex structure
if and only if the linear span of 
  $\{\bar\pa_b \r_1, \ldots, \bar\pa_b \r_{2k+1}\}$ has constant rank $k$ on $\mathcal S.$   
  \end{theorem} 
  \vskip 1pc
  
  \begin{proof}   Suppose that $\mathcal S$ is a complex manifold of complex dimension $n-k$. Then
  \begin{equation}\label{CTS0} \mathbb C\otimes T(\mathcal S) = T^{1,0}(\mathcal S) \oplus T^{0,1}(\mathcal S),\end{equation} where  $T^{1,0}(\mathcal S) :=  H^{1,0}(M) \cap ~\mathbb C\otimes T(\mathcal S) $ and so forth. Each of the direct summand of the right-hand side of (\ref{CTS0}) has rank $n-k.$ 
  For  a local section  $V$ of $H^{0,1}(M)$  it is obvious that  $V\in T^{0,1}(\mathcal S)$ if and only if 
  \begin{equation}\label{drhov2} d\r_j(V)=0, \quad j=1, \ldots, 2k+1.\end{equation} Since $V\in H^{0,1}(M),$  (\ref{drhov2}) is equivalent to 
  \begin{equation*}\label{dbarrho2}\bar\pa_b\r_j(V)=0, \quad j=1, \ldots, 2k+1.\end{equation*}
 We note that there are $n-k$ independent vectors $V_1, \ldots, V_{n-k}$  in $T^{0,1}(\mathcal S). $ 
 Therefore,   the linear span 
\begin{equation}\label{spandbarb3}\langle \bar\pa_b\rho_1, \ldots, \bar\pa_b\rho_{2k+1}\rangle(x) \subset H_x^{*0,1}(M), \quad\forall x\in \mathcal S\end{equation}  has  rank  $n-(n-k)=k.$  

 Conversely,  suppose that (\ref{spandbarb3}) has  rank $k$ at $x\in \mathcal S.$  
Assuming the first $k$ elements of $\{ \bar\pa_b\rho_1, \ldots, \bar\pa_b\rho_{2k+1}\}$ are independent,  we  choose $1$-forms $\o^1, \ldots, \o^{n-k}\in H^{*0,1}(M)$ on a neighborhood $U\subset M$ of $x$ so that 
 ${ H^{*0,1}(M)}$ is spanned by 
\begin{equation*}\label{basis01} \{\bar\pa_b\r_1,\ldots, \bar\pa_b\r_k,  \o^1, \ldots, \o^{n-k} \}.\end{equation*}
Then 
\begin{equation}\label{basis3} \underbrace{\pa_b\r_1,\ldots, \pa_b\r_k, \bar\o^1,\ldots,\bar\o^{n-k}}_{H^{*1,0}(M)},\underbrace{ {\bar\pa_b\r_1,\ldots, \bar\pa_b\r_k,  \o^1,\ldots,\o^{n-k} }}_{H^{*0,1}(M)}, \th \end{equation} is a local basis of  the whole cotangent bundle
 $\mathbb C\otimes T^*(M).$   Let 
\begin{equation}\label{basis4} \underbrace{L_1, \ldots, L_k, \overline V_1, \ldots, \overline V_{n-k}}_{H^{1,0}(M)},\underbrace{ \overline L_1, \ldots, \overline L_k,  V_1, \ldots,  V_{n-k}}_{H^{0,1}(M)}, T\end{equation} be  the dual basis of (\ref{basis3}) for $\mathbb C\otimes T(M).$ 
 Now consider the bundle of maximal complex subspaces of  $T(S)$  given by
\begin{equation*}\label{H(S)}H(S) :=T(S)\cap JT(S) .\end{equation*}   
 Let $\mathcal V(\mathcal S)$ be a sub-bundle of $T(\mathcal S), $  locally defined on the neighborhood $U\cap \mathcal S $ of $x$,   so that  
\begin{equation*}\label{TSdecomposition} T(S)=H(S) \oplus \mathcal V(\mathcal S). \end{equation*} 
Then  the $J$-invariance of $\mathcal S $ comes from the dimension count of  
\begin{equation}\label{CTS} \mathbb C\otimes T(\mathcal S) = H^{1,0}(\mathcal S) \oplus H^{0,1}(\mathcal S) \oplus \left(\mathbb C\otimes \mathcal V(\mathcal S)\right). \end{equation}
To count the dimension of $H^{0,1}(\mathcal S),$  let $V:=X+\sqrt{-1}JX \in H^{0,1}(M), $  $X$ is a real vector,  be any one of $V_1, \ldots, V_{n-k}$ in (\ref{basis4}) and $\r$ be any one of $\r_1, \ldots, \r_{2k+1}.$
We have
\begin{equation*}\label{dbarbrho3}\begin{aligned} 0 =\bar\pa_b \r (V) &= d\r(V)\\
&= d\r(X+\sqrt{-1}JX) \end{aligned}\end{equation*}
which yields
\begin{equation*}\label{dbarbrho4} d\r(X)=0 \quad\text{and }\quad
 d\r (JX) = 0 \end{equation*}  so that  $X$ and $JX$ are tangent to $\mathcal S$.
Therefore,  $V_1, \ldots ,V_{n-k} $ are elements of $H^{0,1}(\mathcal S)$.
Hence $H^{0,1}(\mathcal S)$ has complex dimension $\ge n-k.$ But each fibre in the left-hand side of (\ref{CTS}) has complex dimension $2(n-k)$.
  Therefore,  the last component of the right-hand side of (\ref{CTS}) has rank zero,  which implies that $\mathcal V(\mathcal S) = 0. $    Thus we have 
\begin{equation*}\label{TS=HS} T(\mathcal S) = H(\mathcal S), \end{equation*} which means that $\mathcal S$ is 
$J$-invariant and the conlusion follows from Theorem \ref{NN}.
\end{proof}

If 
$(M, H(M),J)$ is pseudo-convex and $\mathcal S$
    has real dimension $\ge 2q$ then the defining functions $\r_j$'s  of Theorem \ref{complex manifold} can be found from the coefficients and their radicals  of the characteristic polynomial of the Levi-form.  
We adopt from \cite{K} the following 
\begin{definition} Let $\mathcal J $ be a subset of the ring $C^{\infty}(x_0)$  of germs at $x_0\in M$ of smooth 
functions.  The \emph{real radical}  of $\mathcal J$, denoted by $\sqrt[R]{\mathcal J}, $ is the set of all $g\in C^{\infty}(x_0)$ such that there exists an integer $m$ and an $f\in \mathcal J$ so that 
$$|g|^m \le |f| $$ on some neighborhood of $x_0.$
\end{definition}
$\sqrt[R]{\mathcal J} $ is an ideal.  We have   
\begin{corollary}\label{complex manifold2} Let  $M^{2n+1}  $ be a smooth CR manifold with pseudo-convex integrable CR structure $(H(M), J).$   For an integer  $q, $  $0\le q \le n,  $ let $\mathcal S_q$ be the set of points of $M$ where the Levi-form has nullity $\ge q.$   Suppose that the real dimension of  $\mathcal S_q$  is greater than or equal to $2q$
and that 
there is a non-degenerate set of real-valued functions  $\r_1, \ldots , \r_{2(n-q)+1}$
that generates $\sqrt[R]{\{A_0, \ldots, A_{q-1}\}}$ and satisfies 
\begin{equation*}\label{dbarbrho} \text{rank }\langle\bar\pa_b\r_1, \ldots, \bar\pa_b\r_{2(n-q)+1}\rangle = n-k \end{equation*}
on the common  zero set $\mathcal S $ of $\r_j$'s.  Then $\mathcal S $ is  a complex manifold of complex dimension $q$.
\end{corollary}
\begin{proof} On $\mathcal S_q,$ $A_j = 0, $ for all $j=0,1, \ldots, q-1.$  Since $\mathcal S$ is the common zero set of $\r_j$'s  we have  $\mathcal S\subset \mathcal S_q.$    The conclusion follows from  Theorem \ref{complex manifold}. \end{proof}

\begin{remark}
In analytic ($C^{\o}$) category we check a finite set of non-degenerate functions ${\vec \r}:=(\r_1, \r_2, \ldots, \r_{2(n-q)+1} ) $ including the generators of $\sqrt[R]{\{A_0, \ldots, A_{q-1}\}}$
and determine the existence of a complex manifold by showing that 
\begin{equation*}\label{analytic case}\begin{aligned}\bar\pa_b \r_1\we\cdots\we\bar\pa_b \r_{n-q} &\neq 0, \\
 \bar\pa_b \r_1\we\cdots\we\bar\pa_b \r_{n-q} \we \bar\pa_b\r_{\ell}& \equiv 0 \quad\text{mod }  ({\vec \r}),\quad \forall \ell.\end{aligned}\end{equation*}
 \end{remark}


\vskip 2pc
\begin{thebibliography}{9}

 
\bibitem[AH]{AH}
H. Ahn and C.~K. Han, \emph{Local geometry of Levi-forms associated with the existence
of complex submanifolds and the minimality of generic CR manifolds}, J. Geom. Anal.  22 (2012), ~561-- 582.






\bibitem[BCG3]{BCG3}{}
R.~L. Bryant, S.~S. Chern, R.~B. Gardner, H.~L. Goldschmidt, and P.~A.
  Griffiths, \emph{Exterior differential systems}, Mathematical Sciences
  Research Institute Publications, vol.~18, Springer-Verlag, New York, 1991.





\bibitem[CM]{CM} {S. S. Chern and J. K. Moser}, \textit{Real hypersurfaces in complex manifolds},  {Acta Math.}
             {133} (1974) , {219--271}.


\bibitem[CS]{CS}{} So-Chin Chen and Mei-Chi Shaw, 
 \emph{Partial Differential Equations in Several Complex Variables,}  AMS-IP studies in advanced math.  19,  Amer. Math. Soc. , Providence, RI, 2001.
 



\bibitem[DF]{DF} {K. Diederich and J. E. Fornaess}, \textit{Pseudoconvex domains with real-analytic boundary,}   {Ann. of Math. }
             {107} (1978), {371--384}.



\bibitem[NQD]{NQD} N. Q. Dieu, \textit{Zero sets of real polynomials containing complex varieties,}  {Illinois J. Math.} {55-1} (2011), 69--76.




\bibitem[HT]{HT} C.~K. Han and G. Tomassini, \emph{Complex submanifolds in real hypersurfaces,  } J. Korean Math. Soc.
{47}   (2010), 1001--1015.


\bibitem[H]{H}
L. H\"{o}rmander,\emph{An introduction to complex analysis in several variables}, North Holland Publ.,  New York, 1973.





\bibitem[K]{K} J. J. Kohn, \textit{Subellipticity of the $\bar\pa$-Neumann problem on pseudo-convex domains:
Sufficient conditions,}  {Acta Math.}
              {142} (1979), {79--122}.


\bibitem[NN]{NN} A. Newlander and L. Nirenberg, \textit{Complex analytic coordinates in almost complex manifolds,
}  {Ann. of Math.}
              {65} (1957),  {391--404}.





\bibitem[T]{T}  J.-M. Tr\'{e}preau,  \textit{Sur le prolongement holomorphe des fonctions C-R d\'{e}fines sur une hypersurface r\'{e}elle de classe $C^2$ dans $\mathbb C^n,$}  {Invent.  Math.}
             {83} (1986), {583--592}.



\end{thebibliography}
\end{document}